\documentclass[12pt, reqno]{amsart}
\usepackage{amssymb,a4}
\usepackage{amsmath,amsthm,amscd}
\usepackage{amsthm}
\usepackage{amsfonts,amssymb}
\usepackage{mathrsfs}
\usepackage{amsmath}

\def\vbar{\mathchoice{\vrule height6.3ptdepth-.5ptwidth.8pt\kern- .8pt}
{\vrule height6.3ptdepth-.5ptwidth.8pt\kern-.8pt} {\vrule
height4.1ptdepth-.35ptwidth.6pt\kern-.6pt} {\vrule
height3.1ptdepth-.25ptwidth.5pt\kern-.5pt}}

\newtheorem{theorem}{Theorem}[section]
\newtheorem{definition}[theorem]{Definition}
\newtheorem{lemma}[theorem]{Lemma}
\newtheorem{corollary}[theorem]{Corollary}
\newtheorem{prop-def}[theorem]{Proposition-Definition}

\newtheorem{remark}[theorem]{Remark}

\newtheorem{proposition}[theorem]{Proposition}
\numberwithin{equation}{section}

\title[Local derivation on $W(2,2)$]{Local derivations on the Lie algebra $W(2,2)$}

\author{Qingyan Wu}
\address{College of Mathematics and System Sciences, Xinjiang University, Urumqi 830046, Xinjiang, China}
\email{1442648849@qq.com}

\author{Shoulan Gao}
\address{Department of Mathematics, Huzhou University, Zhejiang Huzhou, 313000, China}
\email{gaoshoulan@zjhu.edu.cn}

\author{Dong Liu}
\address{Department of Mathematics, Huzhou University, Zhejiang Huzhou, 313000, China}
\email{liudong@zjhu.edu.cn}
\date{}

\begin{document}

\maketitle

\begin{abstract}
The present paper is devoted to studying local derivations on the Lie algebra $W(2,2)$ which has some outer derivations. Using some linear algebra methods in \cite{CZZ} and a key construction for $W(2,2)$ we prove that every local derivation
on $W(2, 2)$ is a derivation. As an application, we determine all local derivations on the deformed $\mathfrak{bms}_3$ algebra.

\end{abstract}

{\small \textbf{Key words}:  Virasoro algebra; local derivation; $W(2,2)$, deformed $\mathfrak{bms}_3$ algebra

}

\section{Introduction}

The notion of local derivation was originally
introduced by Larson and Sourour aroused from studying the reflexivity of the space of linear maps from an algebra to itself (see \cite{Lar, LarSou}).
Let $L$ be an algebra, $M$ be an $L$-bimodule. A linear mapping $\Delta: L\rightarrow M$ is said to be a local derivation if for every $x$ in $ L$ there exists a derivation $D_{x}: L\rightarrow M$, depending on $x$, satisfying $\Delta(x)=D_{x}(x)$. When $M$ is taken to be $ L$, such a local derivation is called a local derivation on $L$.  Local derivations on various algebras are some kind of local properties for the algebras, which turn out to be very interesting (see \cite{Cr, Bre, BS, P, AK, Kad, LZ}, etc.). Recently, several papers have devoted to studying local derivations for Lie (super)algebras. For examples, it is proved that every local derivation on a finite dimensional semi-simple Lie algebra over an algebraically closed field of characteristic 0 is automatically a derivation in \cite{AK}, and every local derivation
on the Witt algebra is a derivation in \cite{CZZ}.  However, there is no a uniform method to determine all local derivations on Lie algebras, so it is an open question to determine all local derivations for some Lie (super)algebras related to the Virasoro algebra.

The infinite dimensional Lie algebra $W(2,2)$ was first introduced by \cite{ZhangDong} to classify some simple vertex operator algebras and played an important role in many areas of mathematics and physics. Its structure and representation theories were studied in many papers (see \cite{ChenLi, DGL, GaoJiangPei, GJP2, JiangZhang,L,tangw22,Rad}, etc.).
Note that  $W(2,2)$ can be realized as the truncated loop algebra of the Virasoro algebra Vir. More specifically,  $W(2,2)$ is isomorphic to ${\rm Vir}\otimes\mathbb C[t, t^{-1}]/(t^2)$ (\cite{GLZ}). In \cite{GLZ}, the general
truncated loop algebra ${\rm Vir}\otimes\mathbb C[t, t^{-1}]/(t^n), n\ge 1$ was introduced and its quasi-finite weight modules were classified.

We know that $W(2,2)$ has some outer derivations, so it takes many difficulties to determine its local derivations.
Using some linear algebra methods in \cite{CZZ} and a key construction (see Lemma \ref{key-construction1} below), we prove that every local derivation
on $W(2, 2)$ is a derivation in this paper. With this research we can also determine all local derivations on some related Lie (super)algebras. As an application, we prove that every local derivation on the deformed $\mathfrak{bms}_3$ algebra, which corresponds to an infinite dimensional lift of the Maxwell algebra (\cite{CCRS}), is also a derivation. Certainly, such researches can be extended to the general
truncated loop algebra.

The present paper is arranged as follows. In Section 2, we recall some known results and establish some related properties concerning the Lie algebra $W(2,2)$. In Section 3, we determine all local derivations on the Virasoro subalgebra of $W(2,2)$. In Section 4, with a new construction we determine the actions of local derivations on $I_m$, and then prove that every local derivation
on $W(2,2)$ is a derivation. As an application, we prove that every local derivation on the deformed $\mathfrak{bms}_3$ algebra is also a derivation in Section 5.

Throughout this paper, we denote by $\mathbb{Z}$, $\mathbb{C}, \mathbb Z^*, \mathbb C^*$ the sets of  all integers, complex numbers, nonzero integers, nonzero complex numbers, respectively. All algebras are defined over $\mathbb{C}$.

\medskip

\section{Preliminaries}

\medskip

In this section we recall definitions, symbols and establish some auxiliary results for later use in this paper.

A derivation on a Lie algebra $L$ is a linear map
$D: L\rightarrow  L$ which satisfies the Leibniz
law, that is,
$$
D([v,w])=[D(v),w]+[v, D(w)]
$$
for all $v,w\in  L.$  The set of all derivations of $ L$, denoted by $\mathrm{Der}( L)$,  is a Lie algebra with respect to the commutation operation. For $u\in L$, the map
\begin{eqnarray*}
\mathrm{ad}\,u: L\rightarrow L, \ \mathrm{ad}\,u(v)=[u,v], \ \forall v\in L
\end{eqnarray*}
is a derivation and a derivation of this form is called \emph{inner derivation}. The set of all inner derivations of $ L$, denoted by $\mathrm{Inn}( L)$, is an ideal of $\mathrm{Der}( L)$.

Recall that a map $\Delta:  L\rightarrow  L$ is called a \textit{local derivation} if
for every $v\in  L,$ there exists a derivation
 $D_{v}: L\rightarrow  L$ (depending on $v$)
such that
\begin{equation}\label{def1}
\Delta(v)=D_{v}(v). \nonumber\end{equation}

By definition, $W(2,2)$ is an infinite-dimensional Lie algebra with $\mathbb{C}$-basis
\begin{eqnarray*}
\{L_{m},I_{m}, C, C_1|m\in\mathbb{Z}\}
\end{eqnarray*}
and relations
\begin{eqnarray*}
&&[L_{m},L_{n}]=(m-n)L_{m+n}+\delta_{m+n, 0}\frac1{12}(m^3-m)C,\\
&&[L_{m},I_{n}]=(m-n)I_{m+n}+\delta_{m+n, 0}\frac1{12}(m^3-m)C_1,\\
&&[I_{m},I_{n}]=0, [x, C]=[x, C_1]=0, \forall m, n\in\mathbb{Z}, x\in W(2, 2).
\end{eqnarray*}

Clearly the Virasoro algebra ${\rm Vir}:={\rm span}\{L_m, C\mid m\in\mathbb Z\}$ is the subalgebra of $W(2,2)$.
It is well known that Vir is the universal central extension of the Witt algebra $W$, which is the derivation of the Laurent polynomial algebra $\mathbb C[t, t^{-1}]$.
Moreover, $W(2,2)$ can be realized as the truncated loop algebra of the Virasoro algebra.  In fact,  $W(2,2)$ is isomorphic to ${\rm Vir}\otimes\mathbb C[t, t^{-1}]/(t^2)$ (see \cite{GLZ}).

All local derivations on the Witt algebra were determined in \cite{CZZ}.

\begin{theorem}\label{ttm-CZZ}\cite{CZZ}
Every local derivation
on  the Witt algebra $W$ is a derivation.
\end{theorem}

\begin{corollary}\label{coro-Vir}
Every local derivation
on  the Virasoro algebra ${\rm Vir}$ is a derivation.
\end{corollary}
\begin{proof}
Let $\Delta$ be a local derivation of Vir. Since every derivation of the Virasoro algebra is inner (\cite{ZM}),  we have $\Delta(C)=0$. Now by the definition of local derivation and Theorem \ref{ttm-CZZ} we can suppose that $\Delta(L_0)=0$ (also see the proof of Theorem \ref{lemma-vir} below) and $\Delta(L_m)=a_mC$ for some $a_m\in\mathbb C$. By $a_mC=\Delta(L_m)=[u, L_m]$,  we can get $a_m=0$.
\end{proof}

\begin{lemma}\cite{GaoJiangPei}\label{der-w22}
The derivation algebra of $W(2,2)$ is
\begin{eqnarray*}
\mathrm{Der}(W(2,2))=\mathrm{Inn}(W(2,2))\bigoplus\mathbb{C}\delta,
\end{eqnarray*}
where $\delta$ is an outer derivation defined by $\delta(L_{m})=\delta(C)=0,\delta(I_{m})=I_{m}$ and $\delta(C_1)=C_1$ for any $m\in\mathbb{Z}$.
\end{lemma}

\section{Local derivations on the Witt subalgebra}

Set ${\mathcal W}:=W(2,2)/(\mathbb CC\oplus\mathbb CC_1)$, the quotient of $W(2,2)$ by its center. Now we first determine all local derivations on ${\mathcal W}$, and then extend to the Lie algebra $W(2,2)$.
In this section, we shall concern local derivations on the Witt subalgerbra $W$ of ${\mathcal W}$ with the same methods in \cite{CZZ}.

For a local derivation $\Delta:{\mathcal W} \rightarrow {\mathcal W}$ and $x\in {\mathcal W}$,  we always use the symbol $D_{x}$ for the derivation of ${\mathcal W}$ satisfying $\Delta(x)=D_{x}(x)$ and $D_x$ given by Lemma \ref{der-w22} in the following sections.

For a given $m\in\mathbb{Z}^\ast$, recall that $\mathbb{Z}_m=\mathbb{Z}/m\mathbb{Z}$ is the modulo $m$ residual ring of $\mathbb{Z}$.
Then for any $i\in \mathbb{Z}$ we have $\bar i\in \mathbb{Z}_m$, where $\bar i =\{i+km \mid k\in\mathbb{Z}\}$.

 Let $\Delta$ be a local derivation on ${\mathcal W}$ with $\Delta( L_{0})=0 $. For $L_{m}$ with $m\neq0$, set
 \begin{eqnarray}\label{txm2}
\Delta(L_{m})&=&\sum\limits_{n\in\mathbb{Z}}(a_{n}L_{n}+b_{n}I_{n}),
\end{eqnarray}where $a_n, b_n\in\mathbb C$ for any $n\in\mathbb Z$.

Note that $$\mathbb{Z}=\bar0\cup \bar1\cup \cdots \cup \overline{m-1}.$$
Therefore, (\ref{txm2}) can be written as follows:
\begin{eqnarray}\label{equ-Lm1}
\Delta(L_{m})
&=&\sum\limits_{\bar i\in F}\sum\limits_{k=s_i}^{t_i}a_{i+km}L_{i+km}+\sum\limits_{\bar i\in E}\sum\limits_{k=p_{i}}^{q_{i}}b_{i+km}I_{i+km},
\end{eqnarray}
where $s_i\leq t_i\in\mathbb{Z},p_{i}\leq q_{i}\in\mathbb{Z}$, and $E, F\subset\mathbb{Z}_m$.

For $L_{m}+x L_{0}$, where $x\in\mathbb{C}^{*}$, since $\Delta$ is a local derivation, there exists
$\sum\limits_{n\in \mathbb{Z}}(a'_{n}L_{n}+b'_{n}I_{n})\in\mathcal W$,
where $a'_{n},b'_{n}\in\mathbb{C}$ for any $n\in\mathbb Z$, such that
\begin{eqnarray}\label{equ-Lm2}
\nonumber\Delta( L_{m})&=&\nonumber\Delta(L_{m}+xL_{0})\\ \nonumber
&=&\nonumber[\sum\limits_{n\in \mathbb{Z}}a'_{n}L_{n}+b'_{n}I_{n}, L_{m}+x L_{0}]\\ \nonumber
&=&\sum\limits_{\bar i\in F}\sum\limits_{k=s'_i}^{t'_i+1}((i+(k-2)m)a'_{i+(k-1)m}+x(i+km)a'_{i+km})L_{i+km}\\
&&+\sum\limits_{\bar i\in E}\sum\limits_{k=p'_{i}}^{q'_{i}+1}((i+(k-2)m)b'_{i+(k-1)m}+x(i+km)b'_{i+km})I_{i+km}.
\end{eqnarray}

Note that the subset $E, F$ in \eqref{equ-Lm2} is same as that of \eqref{equ-Lm1}.

\begin{lemma}\label{lemma3-1}
Let $\Delta$ be a local derivation on ${\mathcal W}$ such that $\Delta(L_{0})=0$. Then $F=\{\bar0\}$ and $E=\{\bar0\}$ in $(\ref{equ-Lm1})$ and $(\ref{equ-Lm2})$.
\end{lemma}
\begin{proof}
It is essentially same as  that of Lemma 3.2 in \cite{CZZ}.

Assumed that $m\neq0$, $a _{i+s_im}\neq0,a_{i+t_im}\neq0$ and $a_{i+s'_im}^{\prime}\neq0,a_{i+t'_im}^{\prime}\neq0$ for some $\bar i\neq\bar0$. Comparing the right hand sides of (\ref{equ-Lm1}) and (\ref{equ-Lm2}) we see that $s_i=s'_i$ and $t_i\leq t'_i+1$. If $t_i<t'_i+1$, from (\ref{equ-Lm1}) and (\ref{equ-Lm2}), we deduce that
\begin{eqnarray*}
(i+(t'_i-1)m)a_{i+t'_im}^{\prime}=0.
\end{eqnarray*}
Then $i+(t'_i-1)m=0$, i.e., $\bar i=\bar0$, a contradiction. Thus $t_i=t'_i+1$, and $s_i<t_i$.

Comparing (\ref{equ-Lm1}) and (\ref{equ-Lm2}), we deduce that
\begin{eqnarray}\label{txm6}
a_{i+s_im}&=&x(i+s_im)a_{i+s_im}^{\prime}\nonumber ;\\\nonumber
a_{i+(s_i+1)m}&=&(i+(s_i-1)m)a_{i+s_im}^{\prime}+x(i+(s_i+1)m)a_{i+(s_i+1)m}^{\prime};\\\nonumber
&\vdots&\\\nonumber
a_{i+(t_i-1)m}&=&(i+(t_i-3)m)a_{i+(t_i-2)m}^{\prime}+x(i+(t_i-1)m)a_{i+(t_i-1)m}^{\prime};\\\nonumber
a_{i+t_im}&=&(i+(t_i-2)m)a_{i+(t_i-1)m}^{\prime}.
\end{eqnarray}
Since $i+km\neq0$ for $k\in\mathbb{Z}$, eliminating $a_{i+s_im}^{\prime},\cdots,a_{i+(t_i-1)m}^{\prime}$ in this order by substitution we see that
\begin{equation}\label{wqy1}
a_{i+t_im}+\ast x^{-1}+\cdots+\ast x^{-t_i+s_i}=0,
\end{equation}
where $\ast$ are independent of $x$. We always find some $x\in\mathbb{C}^{\ast}$ not satisfying (\ref{wqy1}),  and then get a contradiction. Therefore, $F=\{\bar0\}$. Similarly, $E=\{\bar0\}$. The lemma follows.
\end{proof}

Motivated by Lemma 3.4 in \cite{CZZ}, we have the following lemma.
\begin{lemma}\label{lemma-Lm}
Let $\Delta$ be a local derivation on ${\mathcal W}$ such that $\Delta(L_{0})=0$. Then for any $m\in\mathbb Z^*$, we have
\begin{eqnarray*}
\Delta(L_{m})=c_{m}L_{m}+d_{m}I_{m}
\end{eqnarray*} for some $c_m, d_m\in\mathbb C$.
\end{lemma}
\begin{proof} The proof is essentially same as  that of Lemmas 3.3, 3.4 in \cite{CZZ}.


By Lemma \ref{lemma3-1},   \eqref{equ-Lm1} and \eqref{equ-Lm2}, we have
\begin{eqnarray}\label{wqy2}
\sum\limits_{k=s}^{t}a_{km}L_{km}&=&\sum\limits_{k=s^{\prime}}^{t^{\prime}+1}((k-2)ma_{(k-1)m}^{\prime}+xkma_{km}^{\prime})L_{km}, \label{wqy22}
\end{eqnarray}
where $a_{i+(s'_i-1)m}^{\prime}=a_{i+(t'_i+1)m}^{\prime}
=0$. We may assume that $a _{sm}, a_{tm}, a_{s^{\prime}m}^{\prime}, a_{t^{\prime}m}^{\prime}\neq0$. Clearly, $s^{\prime}\leq s\leq t\leq t^{\prime}+1$, and $s^{\prime}=s$ if $s^{\prime}\neq0$. Our goal is to prove that $s=t=1$, which is essentially same as that of Lemmas 3.3, 3.4 in \cite{CZZ}.

Assume that $s^{\prime}<0$. Then $s^{\prime}=s$. If further $t^{\prime}\geq-1$, from (\ref{wqy2}) we get a set of equations
\begin{eqnarray}\label{wqy3}
a_{sm}&=&xsma_{sm}^{\prime}\nonumber ;\\\nonumber
a_{(s+1)m}&=&(s-1)ma_{sm}^{\prime}+x(s+1)ma_{(s+1)m}^{\prime};\\
&\vdots&\nonumber\\
a_{-m}&=&-3ma_{-2m}^{\prime}-xma_{-m}^{\prime};\nonumber\\
a_{0}&=&-2ma_{-m}^{\prime}.
\end{eqnarray}

If $a_{0}\neq0$, using the same arguments as for (\ref{txm6}), the equations in  (\ref{wqy3}) make a contradiction. So we consider the case that $a_{0}=0$. From  (\ref{wqy3}), we see that $a_{-m}^{\prime}=0$. We continue upwards in (\ref{wqy3}) in this manner to some steps. Then there exists a non-positive integer $l$ such that
\begin{eqnarray*}
a_{0}=a_{-m}=\cdots=a_{lm}=0,a_{(l-1)m}\neq0,a_{-m}^{\prime}=\cdots=a_{(l-1)m}^{\prime}=0.
\end{eqnarray*}
If $s+1<l\, (\leq0)$, then (\ref{wqy3}) becomes
\begin{eqnarray}\label{wqy4}
a_{sm}&=&xsma_{sm}^{\prime}\nonumber ;\\\nonumber
a_{(s+1)m}&=&(s-1)ma_{sm}^{\prime}+x(s+1)ma_{(s+1)m}^{\prime};\\\nonumber
&\vdots&\\
a_{(l-1)m}&=&(l-3)ma_{(l-2)m}^{\prime}.
\end{eqnarray}
Using the same arguments as for (\ref{txm6}), the equations in (\ref{wqy4}) make a contradiction. We need only to consider the case that $s+1=l$, i.e., $l-1=s$. In this case we have that $0=a_{(l-1)m}^{\prime}=a_{s^{\prime}m}^{\prime}\neq0$, again a contradiction. Therefore, $s^{\prime}<0,s^{\prime}=s,t^{\prime}<-1$.

If $t<t^{\prime}+1$, from (\ref{wqy2}) we see that
\begin{eqnarray*}
(t^{\prime}-1)ma_{t^{\prime}m}=0.
\end{eqnarray*}
Then $(t^{\prime}-1)m=0$, i.e., $t^{\prime}=1$, a contradiction. So $t=t^{\prime}+1$ and $s<t$. We get a set of equations from (\ref{wqy2})
\begin{eqnarray}\label{wqy5}
a_{sm}&=&xsma_{sm}^{\prime}\nonumber ;\\\nonumber
a_{(s+1)m}&=&(s-1)ma_{sm}^{\prime}+x(s+1)ma_{(s+1)m}^{\prime};\\\nonumber
&\vdots&\\
a_{tm}&=&(t-2)ma_{(t-1)m}^{\prime}.
\end{eqnarray}
Using the same arguments as for (\ref{txm6}), the equations in (\ref{wqy5}) make a contradiction. Hence $s^{\prime}\geq0$.

If $s^{\prime}\geq1$, then $s=s^{\prime}\geq1$. If $s^{\prime}=0$, then $a_{0}=0$ by (\ref{wqy2}). So $s\geq1$.

Now we have $t\geq s\geq1$. The left is to prove that $t=1$ in (\ref{wqy2}). Otherwise, we assume that $t>1$, and then $t^{\prime}>0$.
\vskip5pt
\noindent{\bf Case 1}: $t^{\prime}>1$.

In this case we can show that $t=t^{\prime}+1$ as in the above arguments. If $s^{\prime}\geq1$, we see that $s=s^{\prime}$ and $s<t$. From (\ref{wqy2}) we obtain a set of (at least two) equations
\begin{eqnarray}\label{wqy6}
a_{sm}&=&xsma_{sm}^{\prime}\nonumber ;\\\nonumber
a_{(s+1)m}&=&(s-1)ma_{sm}^{\prime}+x(s+1)ma_{(s+1)m}^{\prime};\\\nonumber
&\vdots&\\
a_{tm}&=&(t-2)ma_{(t-1)m}^{\prime}.
\end{eqnarray}
Using the same arguments as for (\ref{txm6}), the equations in (\ref{wqy6}) make a contradiction. So $s^{\prime}=0$. Now we have
\begin{eqnarray*}
s^{\prime}=0, s\geq1,t=t^{\prime}+1>2.
\end{eqnarray*}
From (\ref{wqy2}) we obtain a set of (at least two) equations
\begin{eqnarray}\label{wqy9}
a_{2m}&=&2xma_{2m}^{\prime}\nonumber ;\\\nonumber
a_{3m}&=&ma_{2m}^{\prime}+3xma_{3m}^{\prime};\\\nonumber
&\vdots&\\
a_{tm}&=&(t-2)ma_{(t-1)m}^{\prime}.
\end{eqnarray}
Using the same arguments again, the equations in (\ref{wqy9}) make a contradiction.  So this case is not held.
\vskip5pt
\noindent{\bf Case 2}: $t^{\prime}=1$.
In the case $t=2$ and we still have the last equation in  (\ref{wqy9}), which implies that $a_{2m}=0$. It is a contradiction.

Combining with Case 1 and Case 2, we get that $s=t=1$.

Similarly, by Lemma \ref{lemma3-1},   \eqref{equ-Lm1} and \eqref{equ-Lm2}, we also have
\begin{eqnarray}
\sum\limits_{k=p}^{q}b_{km}I_{km}&=&\sum\limits_{k=p^{\prime}}^{q^{\prime}+1}((k-2)mb_{(k-1)m}^{\prime}+xkmb_{km}^{\prime})I_{km}, \nonumber
\end{eqnarray}
where $b_{i+(p'_i-1)m}^{\prime}=b_{i+(q'_i+1)m}^{\prime}
=0$.
By the same considerations as above we can get $p=q=1$. The lemma follows.
\end{proof}

\begin{lemma}\label{lemma-Lm0}
Let $\Delta$ be a local derivation on ${\mathcal W}$ such that $\Delta(L_{0})=\Delta(L_{1})=0$. Then $\Delta( L_{m})=0$ for any $m\in\mathbb{Z}$.
\end{lemma}
\begin{proof}
If $m\geq2$, by Lemma \ref{lemma-Lm}, there exist $c_{m},d_{m}\in\mathbb{C}$ and
\begin{eqnarray*}
\sum\limits_{i\in I}a'_iL_{i}+\sum\limits_{j\in J}b'_{j}I_{j}\in {\mathcal W},
\end{eqnarray*}
where $a_i', b_j'\in\mathbb C$ for any $i\in I, j\in J$, and $I, J$ are finite subsets of $\mathbb Z$, such that
\begin{eqnarray}\label{w1}
c_{m}L_{m}+d_{m}I_{m}\nonumber&=&\Delta( L_{m})=\nonumber\Delta( L_{m}+ L_{1})\\\nonumber
&=&\nonumber[\sum\limits_{i\in I}a'_iL_{i}+\sum\limits_{j\in J}b'_{j}I_{j}, L_{m}+L_{1}].
\end{eqnarray}
Clearly $I, J\subset\{0, 1, m\}$. By easy calculations we have $c_{m}=d_m=0$. Similarly, if $m<0$, we can also get $c_{m}=d_{m}=0$. The proof is completed.
\end{proof}

\begin{theorem}\label{lemma-vir}
Let $\Delta$ be a local derivation on ${\mathcal W}$. Then there exists $D\in {\rm Der}\,{\mathcal W}$ such that $\Delta(L_m)=D(L_m)$ for any $m\in\mathbb Z$.
\end{theorem}
\begin{proof}
Let $\Delta$ be a local derivation on ${\mathcal W}$.  There exists $y\in {\mathcal W}$ such that $\Delta(L_{0})=[y, L_{0}]$. Set $\Delta_{1}=\Delta-\mathrm{ad}(y)$. Then $\Delta_{1}$ is a local derivation such that $\Delta_{1}(L_{0})=0$.  By Lemma \ref{lemma-Lm}, there are $c_{1}, d_{1}\in\mathbb{C}$ such that
\begin{eqnarray*}
\Delta_{1}(L_{1})=c_{1}L_{1}+d_{1}I_{1}.
\end{eqnarray*}
Set $\Delta_{2}=\Delta_{1}+c_{1}\mathrm{ad}(L_{0})+d_{1}\mathrm{ad}(I_{0})$. Then $\Delta_{2}$ is a local derivation such that
\begin{eqnarray*}
\Delta_{2}(L_{0})=0,\Delta_{2}(L_{1})=0.
\end{eqnarray*}
By Lemma \ref{lemma-Lm0},  we have $
\Delta_{2}(L_{m})=0$ for any $m\in\mathbb{Z}$. The theorem holds.
\end{proof}

\section{Local derivations on $W(2,2)$ }

Now we use a key construction to determine $\Delta(I_m)$ and then to determine all local derivations on $W(2,2)$.

Set $L'_{m}=L_{m}+mI_{m}$, we can easily see that  $[L'_{m},L'_{n}]=(m-n)L'_{m+n}$ and $[L'_{m}, I_{n}]=(m-n)I_{m+n}$.
So we can get a new construction of ${\mathcal W}$, which plays a key role in our research.

\begin{lemma}\label{key-construction1}
The subalgebra ${\rm span}_{\mathbb C}\{L_m', I_m\mid m\in\mathbb Z\}$ of $\mathcal W$ is isomorphic to $\mathcal W$.
\end{lemma}

\begin{lemma}\label{lem-Im}
 Let $\Delta$ be a local derivation on ${\mathcal W}$ such that $\Delta(L_m)=0$ for any $m\in\mathbb Z$. Then $\Delta(I_m)\in\mathbb C I_m$.
\end{lemma}
\begin{proof}

By the definition of local derivation and Lemma \ref{der-w22}, we have
\begin{eqnarray}\label{ww222}
\Delta(I_m)\in \oplus_{k\in\mathbb Z}\mathbb CI_k, \ \forall m\in\mathbb Z.
\end{eqnarray}

\noindent{\bf Case 1}: $m\neq0$.

 The proofs of Lemmas \ref{lemma3-1}, \ref{lemma-Lm} yield that $\Delta(L'_{m})=c_{m}L'_{m}+d_{m}I_{m}$ for some $c_m, d_m\in\mathbb C$. Therefore, we see that
\begin{eqnarray}\label{q1q}
\Delta(L'_{m})=\Delta(L_{m}+mI_{m})=\Delta(L_{m})+m\Delta(I_{m})=c_{m}(L_{m}+mI_{m})+d_{m}I_{m}.
\end{eqnarray}
This, together with \eqref{ww222}, yields  $\Delta(I_m)\in\mathbb C I_m$.

\noindent{\bf Case 2}: $m=0$.

Set
\begin{eqnarray*}
\Delta(I_{0})=\sum\limits_{i=s}^{t}a_{i}I_{i},
\end{eqnarray*}
where $s\le t, a_i\in\mathbb C$ for any $i$ and $a_s, a_t\ne 0$.

Choose $p<s, p<0, q>>t$ such that $p+q>t$.
For $I_{0}+L_{p}+L_{q}$, there exist $a_x=\sum_{i\in I}a_i'L_i+\sum_{j\in J}b_j'I_j\in {\mathcal W}, a'\in\mathbb C$, where $a_i', b_j'\in\mathbb C^*$ for any $i\in I, j\in J$, such that
\begin{eqnarray}\label{qq1}
\sum\limits_{i=s}^{t}a_iI_{i}&=&\Delta(I_{0})=\Delta(I_{0}+L_{p}+L_{q})\nonumber\\
&=&[\sum_{i\in I}a_i'L_i+\sum_{j\in J}b_j'I_j, I_{0}+L_{p}+L_{q}]+a'\delta(I_0).
\end{eqnarray}
So $\sum_{i\in I}a_i'L_i=b'(L_{p}+L_{q})$ for some $b'\in\mathbb C$, and
${\rm max}\,\{j\mid b_j'\ne0\}\le q$ and ${\rm min}\,\{j\mid b_j'\ne0\}\ge p$.

In this case we claim that $J\subset\{p, q, 0\}$. In fact,  if $l={\rm max}\{j\in J\mid b_j'\ne 0, j\ne q\}\ge0$, then there exists a nonzero term $b_l'(l-q)I_{q+l}$, where $q+l>t$ in the right hand side of \eqref{qq1}. It is a contradiction. If $l={\rm min}\{j\in J\mid b_j\ne 0, j\ne p\}<0$, then there exists a nonzero term $b_l'(l-p)I_{p+l}$, where $p+l<s$ in the right hand side of \eqref{qq1}. It is also a contradiction. So the claim holds.

Now we can suppose that
\begin{equation}a_x=b'(L_{p}+L_{q})+b_p'I_{p}+b_q'I_{q}+b_0'I_0.\nonumber\end{equation}
Comparing with the coefficients of $I_i, s\le i\le t$, we get $a_i=0$ for any $i\ne 0$.
The proof is completed.
\end{proof}

Now we are in position to get the main result of this paper.

\begin{theorem}\label{main}
Every local derivation
on ${\mathcal W}$ is a derivation.
\end{theorem}
\begin{proof}
Let $\Delta$ be a local derivation on ${\mathcal W}$.  There exists $u\in {\mathcal W}$ such that $\Delta(L_{0})=[u, L_{0}]$. Set $\Delta_{1}=\Delta-\mathrm{ad}(u)$. Then $\Delta_{1}$ is a local derivation such that $\Delta_{1}(L_{0})=0$.  By Lemma \ref{lemma-Lm}, there exist $c_{1}, d_{1}\in\mathbb{C}$ such that
\begin{eqnarray*}
\Delta_{1}(L_{1})=c_{1}L_{1}+d_{1}I_{1}.
\end{eqnarray*}
Set $\Delta_{2}=\Delta_{1}+c_{1}\mathrm{ad}(L_{0})+d_{1}\mathrm{ad}(I_{0})$. Then $\Delta_{2}$ is a local derivation such that
\begin{eqnarray*}
\Delta_{2}(L_{0})=0,\Delta_{2}(L_{1})=0.
\end{eqnarray*}
By Lemma \ref{lemma-Lm0},  we have
\begin{eqnarray*}
\Delta_{2}(L_{m})=0,\forall m\in\mathbb{Z}.
\end{eqnarray*}

By Lemma \ref{key-construction1} and Theorem \ref{lemma-vir}, there exists $D\in {\rm Der}\,{\mathcal W}$ such that $\Delta_2(L_m')=D(L_m')$ for any $m\in\mathbb Z$.
So
$$ \label{delta-0}m\Delta_2(I_m)=D(L_m)+mD(I_m).$$
By Lemma \ref{lem-Im}, we have $\Delta_2(I_m)\in \mathbb CI_m$. So $D=a{\rm ad}\, I_0+b\delta$ for some $a, b\in
\mathbb C$ by Lemma \ref{der-w22}.
In this case
$\Delta_2(I_m)=(b-a)I_m$ for any $m\in\mathbb Z^*$.

Set  $\Delta_3=\Delta_2-(b-a)\delta$ we get
\begin{equation} \label{delta-1}
\Delta_3(L_m)=0,\  \Delta_3(I_n)=0,\  \forall m\in\mathbb Z, n\in\mathbb Z^*.\nonumber
\end{equation}

Now we can suppose that
$\Delta_3( I_0)=cI_0$ for some $c\in\mathbb C$.
So there exist $b_x=\sum_{i\in K}a_k'L_k\in {\mathcal W}, d\in\mathbb C$, where $a_k'\in\mathbb C$ for any $k\in K$, such that
\begin{eqnarray}\label{qq001}
\Delta(I_{0})=\Delta(I_{0}+I_1+I_2)=[\sum_{k\in K}a_k'L_k, I_{0}+I_1+I_2]+d\delta(I_0+I_1+I_2)=cI_0.
\end{eqnarray}
Clearly, $K\subset\{0, 1, 2\}$. By easy calculations in \eqref{qq001} we can get $c=0$.
The proof is completed.
\end{proof}

\begin{corollary}\label{coro-w22}
Every local derivation
on $W(2,2)$ is a derivation.
\end{corollary}
\begin{proof}
From Corollary \ref{coro-Vir}, we just need to consider $C_1$ in the proofs of Theorem \ref{main}.
Let $\Delta$ be a local derivation of $W(2,2)$. By definition we have $\Delta(C_1)=0$. Now from Theorem \ref{main} we can suppose that $\Delta(L_0)=0$ and $\Delta(L_m)=a_mC_1$ and $\Delta(I_m)=b_mC_1$ for some $a_m, b_m\in\mathbb C$. By $a_mC_1=\Delta(L_m)=[u, L_m]$ for some $u\in W(2,2)$,  we can get $a_m=0$.  Similarly, by $b_mC_1=\Delta(I_m)=[w, I_m]+b\delta(I_m)$ for some $w\in W(2,2), b\in\mathbb C$,  we can get $b_m=0$.
\end{proof}

\section{Local derivations on the deformed $\mathfrak{bms}_3$ algebra}

The deformed $\mathfrak{bms}_3$ algebra corresponds to an infinite-dimensional lift of the (2+1)-dimensional Maxwell algebra in the very same way as the $W(2, 2)$  and $\mathfrak{bms}_3$ are infinite-dimensional lifts of the AdS and the Poinca\'re algebras in $2 + 1$ dimensions respectively (\cite{CCRS}).

\begin{definition}\cite{CCRS} The deformed $\mathfrak{bms}_3$ algebra $\widehat{\mathcal B}$ is an infinite-dimensional Lie algebra with $\mathbb{C}$-basis
\begin{eqnarray*}
\{L_{m},J_{m}, I_m, C, C_1, C_2|m\in\mathbb{Z}\}
\end{eqnarray*}
and relations
\begin{eqnarray*}
&&[L_{m},L_{n}]=(m-n)L_{m+n}+\delta_{m+n, 0}\frac1{12}(m^3-m)C,\\
&&[L_{m},J_{n}]=(m-n)J_{m+n}+\delta_{m+n, 0}\frac1{12}(m^3-m)C_1,\\
&&[L_{m},I_{n}]=(m-n)I_{m+n}+\delta_{m+n, 0}\frac1{12}(m^3-m)C_2,\\
&&[J_{m},J_{n}]=(m-n)I_{m+n}+\delta_{m+n, 0}\frac1{12}(m^3-m)C_2,
\end{eqnarray*} for any $m,n\in\mathbb{Z}$, and the others are zero's.
\end{definition}

Clearly the subalgebra span$_{\mathbb C}\{L_m, I_m, C, C_2\mid m\in\mathbb Z\}$ is isomorphic to $W(2,2)$. Moreover, $\widehat{\mathcal B}$ is also a truncated loop algebra of Vir. In fact, $\widehat{\mathcal B}$ is isomorphic to the truncated loop algebra ${\rm Vir}\otimes\mathbb C[t, t^{-1}]/(t^3)$.

\begin{lemma}\label{lemma_2}
The derivation algebra of $\widehat{\mathcal B}$ is
\begin{eqnarray*}
\mathrm{Der}(\widehat{\mathcal B})=\mathrm{Inn}(\widehat{\mathcal B})\bigoplus\mathbb{C}\delta,
\end{eqnarray*}
where $\delta$ is an outer derivation defined by $\delta(L_{m})=\delta(C)=0,\delta(J_{m})=\frac12J_{m}, \delta(I_{m})=I_{m}$ and $\delta(C_i)=C_i, i=1,2$, for any $m\in\mathbb{Z}$.
\end{lemma}
\begin{proof}
It follows by Lemma \ref{der-w22} and some easy calculations.
\end{proof}

Denoted $\mathcal B$ by the centerless $\widehat{\mathcal B}$, i.e.  $\mathcal B:=\widehat{\mathcal B}/(\mathbb CC+\mathbb CC_1+\mathbb CC_2)$.

The following result can be obtained by the essentially same considerations as Sections 3, 4.

\begin{proposition}\label{prop-w22}
Let $\Delta$ be a local derivation on ${\mathcal B}$. Then there exists $D\in {\rm Der}\,{\mathcal B}$ such that $\Delta(L_m)=D(L_m)$ and $\Delta(I_m)=D(I_m)$ for any $m\in\mathbb Z$.
\end{proposition}

Now we use a key construction to determine $\Delta(J_m) (\in {\rm span}\{I_k, J_k\mid k\in\mathbb Z\})$ and then to determine all local derivations on $\mathcal B$ as in Section 4.

Set $L''_{m}=L_{m}+\sqrt2mJ_{m}+m^2I_m$ and $J_m'=J_m+\sqrt2mI_m$, we can easily see that  $[L''_{m},L''_{n}]=(m-n)L''_{m+n}$ and $[L''_{m}, J'_{n}]=(m-n)J'_{m+n}$.
So we have the following result.

\begin{lemma}\label{key-construction2}
The subalgebra ${\rm span}_{\mathbb C}\{L_m'', J_m', I_m\mid m\in\mathbb Z\}$ of $\mathcal B$ is isomorphic to $\mathcal B$.
\end{lemma}

\begin{lemma}\label{lem-w222-Im}
 Let $\Delta$ be a local derivation on ${\mathcal  B}$ such that $\Delta(L_m)=\Delta(I_m)=0$ for any $m\in\mathbb Z$. Then $\Delta(J_m)\in\mathbb C J_m+\mathbb C I_m$.
\end{lemma}
\begin{proof}
\noindent{\bf Case 1}: $m\neq0$.
Using the same considerations as Lemma \ref{lemma-Lm}, we can get that $\Delta(L''_{m})=a_{m}L''_{m}+b_{m}J_{m}'+c_m I_m$.
So \begin{eqnarray*}
&&\Delta(L_{m})+\sqrt2m\Delta(J_{m})+m^2\Delta(I_m)\\
=&&a_{m}(L_{m}+\sqrt2mJ_{m}+m^2I_m)+b_{m}(J_{m}+\sqrt2mI_m)+c_mI_m.
\end{eqnarray*}
Therefore we have $\Delta(J_m)\in\mathbb C J_m+\mathbb C I_m$.

\noindent{\bf Case 2}: $m=0$. It is essentially same as that of Lemma \ref{lem-Im} in Section 4.
\end{proof}

\begin{lemma}\label{lem-Jm}
 Let $\Delta$ be a local derivation on ${\mathcal B}$ such that $\Delta(L_m)=\Delta(I_m)=0$ for any $m\in\mathbb Z$. Then $\Delta(J_m)=0$ for any $m\in\mathbb Z$.
\end{lemma}
\begin{proof}

For any $m\in\mathbb Z$,
 $\Delta(J_m)=a_mJ_m+b_mI_m$ for some $a_m, b_m\in\mathbb C$  by Lemma \ref{lem-w222-Im}.

For $J_m+L_1+L_2$, there exist $\sum_{i\in I}a_i'L_i+\sum_{j\in J}b_j'J_j+\sum_{k\in K}c_k'I_k\in {\mathcal B}, b'\in\mathbb C$, where $a_i', b_j', c_k'\in\mathbb C$ for any $i\in I, j\in J, k\in K$, such that
\begin{eqnarray*}&&a_mJ_m+b_mI_m=\Delta(J_m)=\Delta(J_m+L_1+L_2)\\
&=&[\sum_{i\in I}a_i'L_i+\sum_{j\in J}b_j'J_j+\sum_{k\in K}c_k'I_k, J_m+L_1+L_2]+b'\delta(J_m+L_1+L_2).
\end{eqnarray*}
 It is clear that
$\sum_{i\in I}a_i'L_i=a'(L_1+L_2)$,  and then $J\subset\{m\}$ and $K=\emptyset$. By easily calculation we get $b_m=0$.

For $J_m+I_{2m}$, there exist $\sum_{i\in I'}f_i'L_i+\sum_{j\in J'}d_j'J_j\in\mathcal B, c'\in\mathbb C$, where $f_i', d_j'\in\mathbb C$ for any $i\in I', j\in J'$,  such that
$a_mJ_m=\Delta(J_m)=\Delta(J_m+I_{2m})=[\sum_{i\in I'}f_i'L_i+\sum_{j\in J'}d_j'J_j, J_m+I_{2m}]+c'\delta(J_m+I_{2m})$. It can easily get $I'\subset\{0, m\}, J'\subset\{m, 2m\}$.
So \begin{eqnarray*}&&a_mJ_m=\Delta(J_m+I_{2m})\\
&=&[f_0'L_0+f_m'L_m+d_m'J_m+d_{2m}'J_{2m}, J_m+I_{2m}]+c'\delta(J_m+I_{2m})\\
&=&-mf_0'J_m-2mf_0'I_{2m}-mf_m'I_{3m}+md_{2m}'I_{3m}+\frac12c'J_m+c'I_{2m}\\
&=&\frac12(c'-2mf_0')J_m+(c'-2mf_0')I_{2m}+m(d_{2m}'-f_m')I_{3m}.
\end{eqnarray*}
So we have $c'-2mf_0'=0$ and then $a_m=0$.
The lemma follows.
\end{proof}

Now we are in position to get the main result of this section.

\begin{theorem}\label{main2}
Every local derivation
on ${\mathcal B}$ is a derivation.
\end{theorem}
\begin{proof}
Let $\Delta$ be a local derivation on ${\mathcal B}$.  By Proposition \ref{prop-w22}, there exists $D\in {\rm Der}\, \mathcal B$ such that $\Delta(L_m)=D(L_m)$ and $\Delta(I_m)=D(I_m)$ for any $m\in\mathbb Z$.

Replaced $\Delta$ by $\Delta-D$, it follows
\begin{eqnarray*}
\Delta(L_{m})=\Delta(I_m)=0,\forall m\in\mathbb{Z}.
\end{eqnarray*}
Then the theorem follows from Lemma \ref{lem-Jm}.
\end{proof}

\begin{corollary}\label{ttm-tang}
Every local derivation
on the deformed $\mathfrak{bms}_3$ algebra $\widehat{\mathcal B}$  is a derivation.
\end{corollary}
\begin{proof}
It is essentially same as that of Corollary \ref{coro-w22}.
\end{proof}

\begin{remark}
Based on the above researches we can get such results for the general truncated Virasoro algebra ${\rm Vir}\otimes\mathbb C[t, t^{-1}]/(t^n)$ for any $n\ge 1$ by induction. Moreover,
we can determine local derivations on some related algebras, such as the super Virasoro algebra and the super $W(2,2)$ algebra (see \cite{WGL}).
\end{remark}

\section*{Acknowledgments}
This work is partially supported by the NNSF (Nos. 12071405, 11971315, 11871249)

\end{document}